\documentclass[a4paper]{amsart}
\usepackage[utf8]{inputenc}
\usepackage{graphicx}
\usepackage{capt-of}

\usepackage[T1]{fontenc}
\usepackage{lmodern}

\DeclareUnicodeCharacter{2212}{-}

\usepackage{amsmath, amssymb}
\usepackage{amsthm}
\usepackage{enumerate}
\usepackage{mathtools}
\usepackage{tikz-cd}
\usepackage{comment}
\usepackage{color}
\usepackage{hyperref}
\usepackage{soul}

\newcommand{\p}{\operatorname{\mathbb{P}}}
\newcommand{\C}{\operatorname{\mathbb{C}}}
\newcommand{\Z}{\operatorname{\mathbb{Z}}}

\newcommand{\kk}{{\bf k}}

\newcommand{\id}{\operatorname{id}}

\newcommand{\car}{\operatorname{char}}

\newcommand{\Bir}{\operatorname{Bir}}
\newcommand{\Cr}{\operatorname{Cr}}

\newcommand{\Div}{\operatorname{Div}}

\usepackage[shortlabels]{enumitem}
\setlist[enumerate]{label=\rm{(\arabic*)}}
\setlist[enumerate,2]{label=\rm({\it\roman*})}
\setlist[itemize]{label=\raisebox{0.25ex}{\tiny$\bullet$}}

\theoremstyle{plain}
\newtheorem{theorem}{Theorem}[section]
\newtheorem{lemma}[theorem]{Lemma}

\newtheorem{corollary}[theorem]{Corollary}

\theoremstyle{definition}

\newtheorem{remark}{Remark}[section]

\author{Anthony Genevois} 
\author{Anne Lonjou}
 
 \author{Christian Urech}

\newcommand{\Address}{{
		\bigskip
		\small
		
		\textsc{Institut Montpellierain Alexander Grothendieck, 499-554 Rue du Truel, 34090 Montpellier, France.}\par\nopagebreak
		\textit{E-mail address}: \texttt{anthony.genevois@umontpellier.fr}
		\medskip
		
		\textsc{Department of Mathematics, University of the
			Basque Country UPV/EHU, Leioa, Spain. IKERBASQUE, Basque Foundation for Science, Bilbao, Spain.}\par\nopagebreak
		\textit{E-mail address}: \texttt{anne.lonjou@ehu.eus}
		\medskip
		
		\textsc{ETH Z\"urich, Department of Mathematics, R\"amistrasse 101, 8092 Z\"urich, Switzerland.}\par\nopagebreak
		\textit{E-mail address}: \texttt{christian.urech@math.ethz.ch}
		\medskip
		
}}

\title{On a Theorem by Lin and Shinder through the Lens of Median Geometry}

\begin{document}
\maketitle

\begin{abstract}
 Recently, Lin and Shinder constructed non-trivial homomorphisms from Cremona groups of rank $>2$ to $\mathbb{Z}$ using motivic techniques. In this short note we propose an alternative perspective from median geometry on their theorem. 
\end{abstract}

\section{Introduction}
\noindent
To an algebraic variety $X$ over a field $\kk$ we can associate its group of birational transformations $\Bir(X)$. The group $\Cr_n(\kk):=\Bir(\p^n_\kk)$ is called the \emph{Cremona group in $n$ variables}. The classical question, whether $\Cr_2(\kk)$ is a simple group has been answered in the negative in the last decades \cite{Cantat-Lamy}, \cite{Lonjou_non_simplicity}, \cite{shepherd2021some}, \cite{lamy2020signature}. In fact, the group is highly non-simple in the sense that all its simple subgroups are very small \cite{urech2020simple}. However, if $\kk$ is a perfect field, $\Cr_2(\kk)$ is generated by involutions \cite{lamy2021generating} and hence does not admit non-trivial homomorphisms to $\Z$. 

If $n\geq 3$ the situation is understood less well. In \cite{blanc2021quotients} it has been shown that $\Cr_n(\kk)$ is never simple if $\kk$ is a subfield of $\C$ and $n\geq 3$. In fact, the authors show that in this case, $\Cr_n(\kk)$ surjects to an infinite free product of cyclic groups of order two. In \cite{lin2022motivic} the authors construct surjective homomorphisms from $\Cr_n(\kk)$ to $\Z$ for $n\geq 3$ for many fields $\kk$ (see below); it still seems to  be an open question whether $\Cr_3(\C)$ surjects to $\Z$ though. The results in \cite{lin2022motivic} are proved in the context of Grothendieck rings of varieties using motivic techniques. In \cite{blanc2022birational} the result of Lin and Shinder is reproved using an enhanced version of the techniques developed in \cite{blanc2021quotients}. In fact, the authors show the much stronger result that $\Cr_n(\kk)$ surjects to a free product of infinitely many copies of $\Z$, if $n\geq 4$ and $\kk\subset \C$. Nevertheless, in \cite{lin2022motivic} the authors give a stronger constraint on the generators of $\Cr_n(\kk)$; namely they show that the groups $\Cr_n(\kk)$ are not generated by pseudo-regularisable transformations. Meanwhile, in \cite{blanc2021quotients} the authors prove only that the groups $\Cr_n(\kk)$ are not generated by elements of finite order. An element $f\in\Bir(X)$ is \emph{pseudo-regularisable} if there exists a variety $Y$ and a birational map $\psi\colon Y\dashrightarrow X$ such that the exceptional loci of $\psi^{-1}f\psi$ and its inverse have codimension strictly greater than 1.

\medskip \noindent 
In \cite{lin2022motivic}, the authors define a map $c\colon\Bir(X)\to \Z[\Bir_{n-1}/\kk]$, where $\Z[\Bir_{n-1}/\kk]$ is the free abelian group on the set of $(n − 1)$-dimensional birational equivalence classes of varieties over $\kk$, by $c(f)=[K_1]+\cdots +[K_m]-[H_1]-\cdots-[H_k]$ in the notations of Theorem~\ref{thm:quotient}. They then use motivic arguments in order to show that $c(f)$ is a homomorphism.

	\medskip \noindent 
Recently, there have been several articles about Cremona groups acting isometrically on median graphs \cite{lonjou-urech}, \cite{10.1093/imrn/rnad015}, \cite{lonjou2023finitely}\footnote{	Many papers, including \cite{lonjou-urech} rather use the notion of CAT(0) cube complexes instead of median graphs. Median graphs are exactly the 1-skeletons of CAT(0) cube complexes, so median graphs and CAT(0) cube complexes essentially define the same objects. However, it is for various reasons more natural to directly work with median graphs (see \cite{genevois2023cat} for references and a discussion). }. The goal of this note is to give an independent proof of the theorem of Lin and Shinder using the actions of $\Bir(X)$ on median graphs constructed by the second and third author in \cite{lonjou-urech}.

\medskip \noindent
We prove an a priori stronger version of their result. Recall that two subvarieties $A,B\subset X$ are \emph{Cremona equivalent} if there exists a birational transformation $f$ of $X$ such that $f(A)=B$, where $f(A)$ denotes the strict transform of $A$. Denote by $\Z[\Div(X)/_\approx]$ the free abelian group on the set of Cremona equivalence classes of $(n-1)$-dimensional subvarieties of $X$.   Our main result is the following:

\begin{theorem}\label{thm:quotient}
	Let $X$ be a normal variety and let $f\in\Bir(X)$. Let $H_1,\dots, H_k$ be the irreducible components of strict codimension 1 of the exceptional locus of $f$ and let $K_1,\dots, K_m$ be the irreducible components of strict codimension 1 of the exceptional locus of $f^{-1}$. The assignment  $f\mapsto [K_1]+\dots +[K_m] -[H_1]-\dots-[H_k]$ defines a group homomorphism $ \varphi\colon\Bir(X)\to \Z[\Div(X)/_\approx]$. 
	
	\medskip \noindent
	In particular, if there exists no bijection $\Phi\colon\{H_1,\dots, H_k\}\to\{K_1,\dots, K_m\}$
	such that $H_i$ is Cremona equivalent to $\Phi(H_i)$ for all $i$, then $f\notin\ker(\varphi)$.
	Moreover, $\ker(\varphi)$ contains all pseudo-regularisable transformations. In particular, $\Bir(X)$ is not generated by pseudo-regularisable transformations.
\end{theorem}
\noindent
Note that the condition in the second part of Theorem~\ref{thm:quotient} is in particular satisfied if there exists an $H_i$ that is not Cremona equivalent to any of the $K_j$.
The construction of our {group} homomorphism $\varphi$ relies on a general construction of homomorphisms from groups acting without inversion on median graphs to right-angled Artin groups.

\medskip \noindent
The existence of birational transformations whose exceptional loci satisfy the condition from the second part of Theorem~\ref{thm:quotient} follows for instance from an example by Hassett and Lai \cite{hassett2018cremona} for the case $\p^n_\kk$ if $n\geq 4$. In \cite{lin2022motivic}, the work of Hassett and Lai and other examples are extended in order to construct birational maps $f$ such that $c(f)\neq 0$. Note that $c(f)\neq 0$ implies in particular that the condition on the exceptional locus of $f$ from Theorem~\ref{thm:quotient} is satisfied.  The following theorem summarizes what is known so far:

\begin{theorem}[{\cite[Theorem~4.4]{lin2022motivic}}]\label{thm:linshinder}
	There exist birational transformations $f\in\Bir(\p^n_{\kk})$ such that $c(f)\neq 0$ in the following cases:
	\begin{enumerate}
		\item $n\geq 5$ and $\kk$ infinite,
		\item $n\geq 4$ and $\car(\kk)=0$,
		\item $n=3$ and $\kk$ a number field, a function field over a number field, or a function field over an algebraically closed field.
	\end{enumerate} 
\end{theorem}

\begin{remark}
	In \cite{lin2022motivic} the authors ask for Case (2) in addition that $\kk\subset\C$. However, this condition is not needed. Indeed, if two varieties $X$ and $Y$ are not birationally equivalent over the algebraic closure of their field of definition, they are not birationally equivalent over any field, as can be seen by a spread out argument. 
\end{remark}

\noindent
With Theorem~\ref{thm:linshinder}, our Theorem~\ref{thm:quotient} has the following direct corollary (which in particular does not rely on the motivic fact that $c$ is a homomorphism):

\begin{corollary}\label{cor:shinderlin}
	There exists a non-trivial homomorphism $\varphi\colon\Bir(\mathbb{P}_{\kk}^n)\to\Z$ in the following cases:
		\begin{enumerate}
		\item $n\geq 5$ and $\kk$ infinite,
		\item $n\geq 4$ and $\car(\kk)=0$,
		\item $n=3$ and $\kk$ a number field, a function field over a number field, or a function field over an algebraically closed field.
	\end{enumerate} 
	 Moreover, in these cases, $\Bir(\mathbb{P}_{\kk}^n)$ is not generated by pseudo-regularisable transformations. 
\end{corollary}

\noindent
Let us observe that being Cremona equivalent is a weaker notion than being birationally equivalent  for divisors \cite{mella2012equivalent}. The homomorphism $c$ factors as
\[
\Bir(X)\xrightarrow{\varphi}\Z[\Div(X)/_\approx]\xrightarrow{\alpha} \Z[\Bir_{n-1}/\kk],
\]
\noindent
where $\alpha$ associates to a hypersurface in $X$ its birational equivalence class.
 A priori, Theorem~\ref{thm:quotient} could be used to construct non-trivial homomorphisms from $\Bir(X)\to \Z[\Div(X)/_\approx]$ for some varieties $X$, even if the homomorphism $c$ on $\Bir(X)$ is trivial. For this, it would be interesting to investigate whether the restriction of $\alpha$ to the image of $\varphi$ is injective or not.
 \medskip
 
\paragraph{\bf Acknowledgments.} The authors thank J\'er\'emy Blanc, Yves Cornulier, Evgeny Shinder and the referee for their comments that helped improve the exposition of the paper. The second author was partially supported by the Basque Government grant IT1483-22, by MCIN /AEI /10.13039/501100011033 / FEDER through the Spanish Governement grant PID2022-138719NA-I00, by H2020-MSCA-COFUND-2020-101034228-WOLFRAM2, and by the French National grant through the project GOFR ANR-22-CE40-0004. The third author was partially supported by the SNSF grant 200020\_192217 "Geometrically ruled surfaces".

\section{Median graphs}

\noindent

\noindent
Recall that a connected graph $\mathcal{G}$ is \emph{median} if, for all vertices $x_1,x_2,x_3 \in \mathcal{G}$, there exists a unique vertex $m \in \mathcal{G}$ satisfying
$$d(x_i,x_j)= d(x_i,m)+d(m,x_j) \text{ for all } i \neq j,$$
where $d(x,y)$ denotes the length of the shortest path between two vertices $x$ and $y$. A path from $x$ to $y$ of length exactly $d(x,y)$ is called a \emph{geodesic}.
The vertex $m$ is the \emph{median point} of $x_1,x_2,x_3$. 

\medskip \noindent
Median graphs appeared in geometric group theory through \emph{CAT(0) cube complexes}, which are, roughly speaking, cellular complexes obtained by gluing cubes together and endowed with a nonpositively curved metric. Introduced in \cite{Gromov} as a convenient source of CAT(0) spaces, the combinatorial structure of CAT(0) cube complexes was highlighted in \cite{MR1347406} thanks to \emph{hyperplanes}, which we define below. A few years after, it was realised that CAT(0) cube complexes essentially defined the same objects as median graphs, introduced in metric graph theory during the 1960s \cite{MR2405677, MR2798499}. More precisely, a graph is median if and only if it is the one-skeleton of a CAT(0) cube complex \cite{mediangraphs, MR1663779, Roller}. In practice, this is the point of view that is the most used nowadays. In the last few decades, these spaces became a classical tool in geometric group theory. Indeed, on the one hand, many groups of interest turn out to admit rich actions on median graphs (or equivalently, CAT(0) cube complexes); and, on the other hand, median geometry provides powerful tools that help us extracting useful information from such actions. As an illustration of this philosophy, a famous recent achievement is given by the proof of the virtual Haken conjecture in low-dimensional topology \cite{MR3104553}.

\medskip \noindent
As already mentioned, a fundamental tool in the study of median graphs is the notion of \emph{hyperplanes}. A \emph{hyperplane} in a median graph $\mathcal{G}$ is an equivalence class of edges with respect to the reflexive-transitive closure of the relation that identifies two opposite edges in a $4$-cycle, i.e. one can think of hyperplanes as parallelism classes of edges.  A path in $\mathcal{G}$ is said to \emph{cross} a hyperplane, if it contains an edge from the equivalence class. Two vertices are \emph{separated} by a hyperplane if every path connecting them crosses it. 
Two  distinct hyperplanes are \emph{transverse} if they contain two adjacent edges in some $4$-cycle (see Figure \ref{fig:transverse}). Two distinct hyperplanes are either transverse or disjoint. \noindent

\medskip\noindent
Let us recall the following facts (see \cite{MR1347406}):
\begin{itemize}
		\item The graph obtained from $\mathcal{G}$ by removing the (interiors of the) edges in a hyperplane $J$ has exactly two connected components, referred to as \emph{halfspaces}, which are convex, i.e.\ all geodesics between two vertices in a given halfspace are contained in the same halfspace.
	\item The \emph{carrier} of a hyperplane $J$, i.e.\ the smallest induced (or full) subgraph containing $J$, is convex.
	\item A path in $\mathcal{G}$ is geodesic if and only if it crosses each hyperplane at most once.
	\item The distance between two vertices coincides with the number of hyperplanes separating them. 
\end{itemize}

	\begin{figure}
	\includegraphics[width=0.5\linewidth]{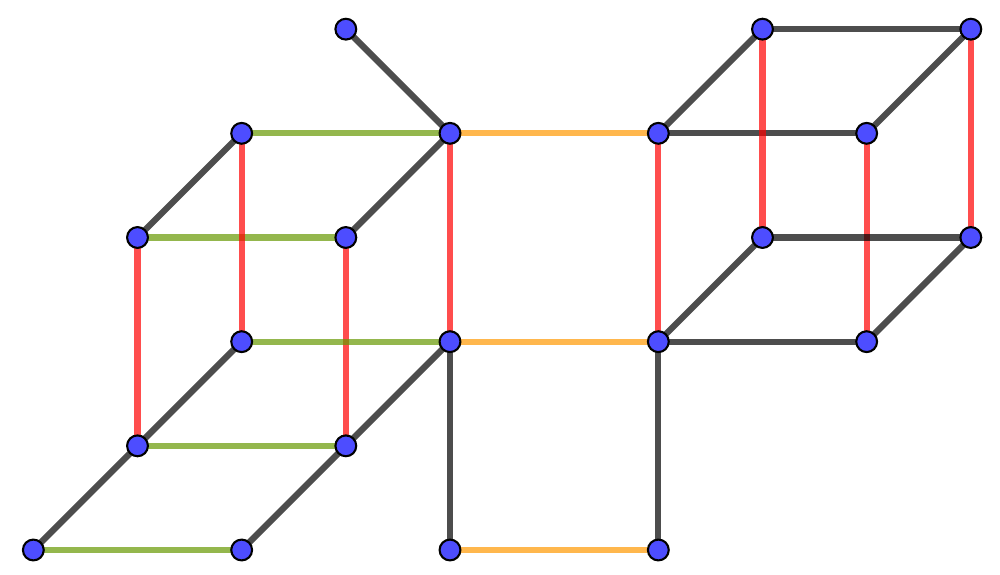}
	\caption{In this example, the hyperplane defined by the red edges is transverse to the hyperplanes defined by the green and yellow edges.\label{fig:transverse}}
\end{figure}

\noindent
Let $\gamma:=(x_0, \ldots, x_n)$ be a path. We say that $\gamma':=(x_0,\ldots x_{i-1},x_i',x_{i+1},\ldots, x_n)$ is obtained from $\gamma$ by \emph{flipping a $4$-cycle}, if the vertices $x_{i-1},x_i,x_{i+1},x_i'$ form an induced $4$-cycle.

\begin{lemma}\label{lem:PathMedian}
	Let $\mathcal{G}$ be a median graph and let $\alpha,\beta$ be two paths with the same endpoints. Then $\alpha$ can be transformed into $\beta$ by adding or removing backtracks and by flipping $4$-cycles.
\end{lemma}

\noindent
This property is a consequence of the fact that filling the $4$-cycles of a median graph with squares yields a simply connected square complex. For the reader's convenience, we include below a short proof based on median geometry.

\begin{proof}[Proof of Lemma~\ref{lem:PathMedian}.]
	First, we claim that $\alpha$ can be made geodesic by removing backtracks and flipping $4$-cycles. Indeed, if $\alpha$ is not geodesic then it crosses some hyperplane $J$ twice. Let $\alpha_0$ denote a subsegment of $\alpha$ lying between two edges of $J$. Without loss of generality, we assume that $\alpha_0$ does not cross a hyperplane twice. Consequently, $\alpha_0$ is a geodesic. By convexity of halfspaces and carriers, $\alpha_0$ has a mirror path $\alpha_1$ on the other side of $J$.
	\begin{center}
		\includegraphics[width=0.6\linewidth]{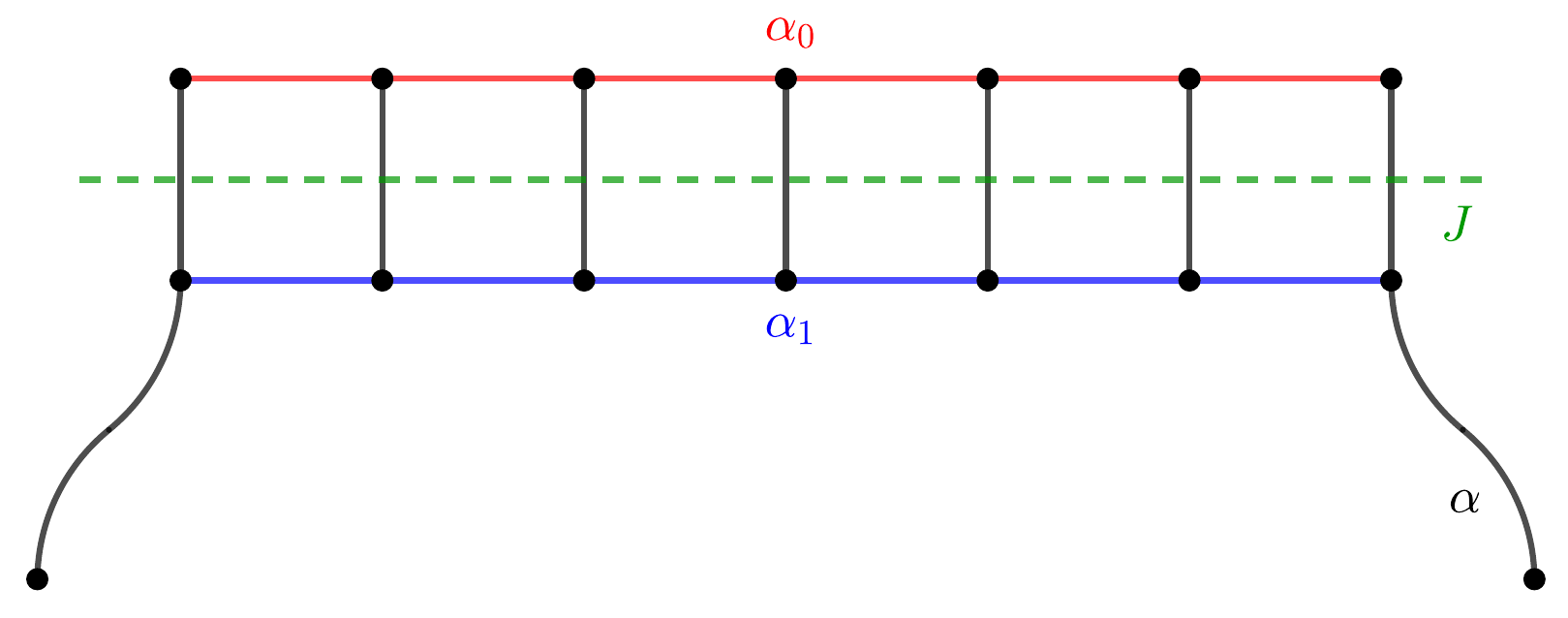}
	\end{center}
	
	\noindent
	Clearly, we can replace $\alpha_0$ in $\alpha$ with $\alpha_1$ by flipping $4$-cycles and removing one backtrack. The process decreases the length of $\alpha$. After finitely many iterations, we obtain a geodesic. This proves our first claim.
	
	\medskip \noindent
	Now, assuming that $\alpha$ is a geodesic, we fix an arbitrary geodesic $\zeta$ connecting the endpoints of $\alpha$, say $x$ and $y$, and we claim that $\alpha$ can be transformed into $\zeta$ by flipping $4$-cycles. We argue by induction on the length of $\alpha$. Let $a \in \alpha$ and $b \in \zeta$ denote the respective neighbours of $x$. If $a=b$, then we obtain the result by the induction hypothesis. Otherwise, notice that the median point $m$ of $y,a,$ and $b$ yields the fourth vertex of a $4$-cycle spanned by $x,a,b$. 
	\begin{center}
		\includegraphics[width=0.5\linewidth]{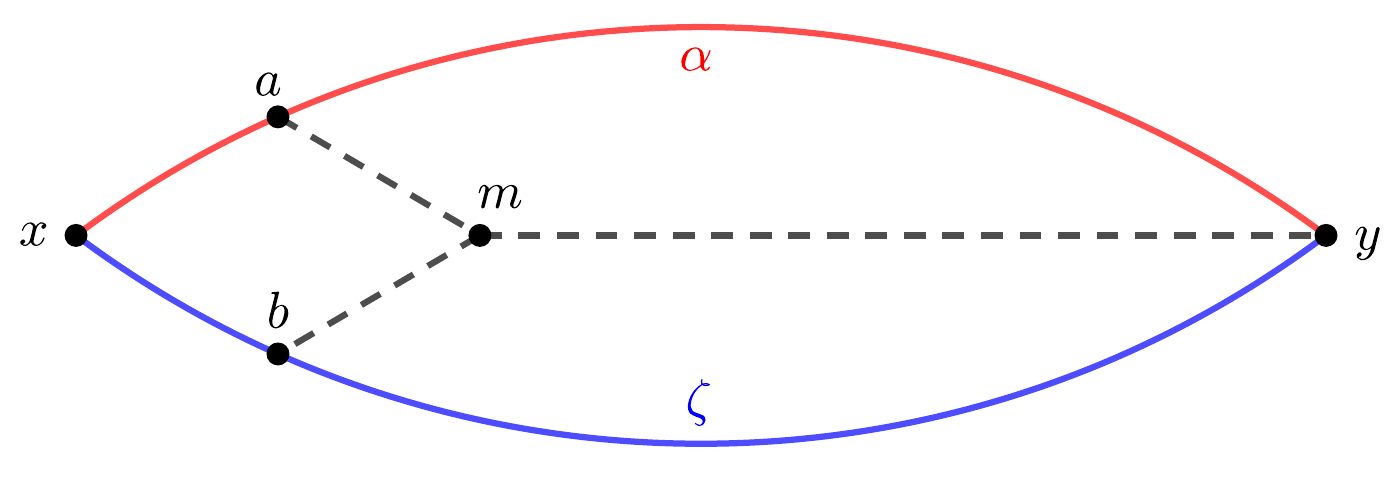}
	\end{center}
	
	\noindent
	As a consequence of our induction hypothesis, $\alpha$ can be transformed as a geodesic of the form $[x,a] \cup [a,m] \cup [m,y]$. By flipping a $4$-cycle, we get the path $[x,b] \cup [b,m] \cup [m,y]$. By applying our induction hypothesis one more time, we can transform this path into~$\zeta$. 
\end{proof}

\section{Homomorphisms to right-angled Artin groups}\label{sec:hom}
\noindent
Given a graph $\Gamma$, the \emph{right-angled Artin group} $A(\Gamma)$ is defined by the presentation
$$\langle u \in V(\Gamma) \mid  uvu^{-1}v^{-1}=1 \ (\{u,v\} \in E(\Gamma)) \rangle$$
where $V(\Gamma)$ and $E(\Gamma)$ denote the vertex- and edge-sets of $\Gamma$. Right-angled Artin groups interpolate between free groups (when $\Gamma$ has no edges) and free abelian groups (when $\Gamma$ is complete). Given an arbitrary group $G$ acting on a median graph $\mathcal{G}$ with no \emph{hyperplane inversion} (i.e. with no isometry stabilising a hyperplane and switching its two halfspaces), our goal now is to construct a natural morphism from $G$ to some right-angled Artin group.  For free actions, this is done in \cite{MR2377497}. Here, we propose a construction for arbitrary actions, inspired by \cite{MR4586831}. \cite{MR4586831} deals more generally with actions on quasi-median graphs and morphisms to graph products of groups. When restricted to actions on median graphs, it yields morphisms to right-angled Coxeter groups. Below, we explain how to adapt the construction in order to get morphisms to right-angled Artin groups instead.

\medskip \noindent
We fix once for all an orientation of the hyperplanes of $\mathcal{G}$ (i.e. we orient the edges of $\mathcal{G}$ so that two opposite sides in a $4$-cycle have parallel orientations).
We denote by $\mathcal{H}$ the set of hyperplanes of $\mathcal{G}$.
Let $\Gamma$ denote the graph whose set of vertices is $\mathcal{H}/G$, i.e. the set of the $G$-orbits of hyperplanes in $\mathcal{G}$ and whose edges connect two orbits whenever they contain transverse hyperplanes. Notice that an oriented path $\alpha$ in $\mathcal{G}$ is naturally labelled by the word written over $V(\Gamma) \sqcup V(\Gamma)^{-1}$ given by the oriented hyperplanes successively crossed by $\alpha$, as $G$ does not have hyperplane inversion.

\begin{theorem}
Let $G$ be a group acting on a median graph $\mathcal{G}$ without hyperplane inversion. Let $o$ be a vertex of $\mathcal{G}$. The following map
$$\Theta : \left\{ \begin{array}{ccc} G & \to & A(\Gamma) \\ g & \mapsto & \text{label of a path from $o$ to $g \cdot o$} \end{array} \right.$$
defines a group homomorphism.
\end{theorem}

\begin{proof}
The fact that $\Theta$ is well-defined, i.e.\ the fact that the element of $A(\Gamma)$ represented by the label of a path only depends on its endpoints follows from Lemma~\ref{lem:PathMedian}. Indeed, adding or removing a backtrack to a path amounts to adding or removing a subword $uu^{-1}$ or $u^{-1}u$ (where $u \in V(\Gamma)$) to its label. And flipping a $4$-cycle amounts to replacing a subword $uv$ (resp. $u^{-1}v$, $uv^{-1}$, $u^{-1}v^{-1}$) with $vu$ (resp. $vu^{-1}$, $v^{-1}u$, $v^{-1}u^{-1}$) where $\{u,v\} \in E(\Gamma)$. Thus, $\Theta$ is well-defined.

\medskip \noindent
In order to see that $\Theta$ is a group homomorphism, let $g,h \in G$ be two elements and fix two oriented paths $[o,go]$ and $[o,ho]$. Then $[o,go] \cup g [o,ho]$ yields an oriented path from $o$ to $gho$, and the label of $g[o,ho]$ coincides with the label of $[o,ho]$ because two edges in the same $G$-orbit have the same label. Hence
$$\begin{array}{lcl} \Theta(gh) & = & \text{label of } [o,go] \cup g [o,ho] = (\text{label of } [o,go]) \cdot (\text{label of } g[o,ho])  \\ \\ & = & (\text{label of } [o,go]) \cdot (\text{label of } [o,ho]) = \Theta(g) \Theta(h). \end{array}$$\qedhere
\end{proof}

\noindent Let us observe that the group homomorphism $\Theta$ does not fundamentally depend on the choice of the basepoint $o$.
\begin{lemma}
Let $\Theta'$ denote the group homomorphism defined as $\Theta$ but with respect to a new basepoint $o'$. $\Theta$ and $\Theta'$ only differ by a conjugation.
\end{lemma}

\begin{proof}
Fix a path $[o',o]$ connecting $o'$ to $o$. We denote by $[o,o']$ the path connecting $o$ to $o'$ obtained by travelling along $[o',o]$ in the reverse direction. For every element $g$ of our group, given an arbitrary path $[o,go]$ connecting $o$ to $go$, we have
$$\begin{array}{lcl} \Theta'(g) & = & \text{label of } [o',o] \cup [o,go] \cup g[o,o'] \\ \\ & = & (\text{label of } [o',o]) \cdot (\text{label of } [o,go]) \cdot (\text{label of } g[o,o']) \\ \\ & = & (\text{label of } [o',o]) \cdot (\text{label of } [o,go]) \cdot (\text{label of } [o,o']) \\ \\ & = & (\text{label of } [o',o]) \cdot \Theta(g) \cdot (\text{label of } [o',o])^{-1}. \end{array}$$
\end{proof}

As a consequence, if our right-angled Artin group $A(\Gamma)$ turns out to be abelian, which will be the case in the next section, then our morphism $\Theta$ does not depend on the choice of the basepoint $o$. As another consequence, we also note that all $g\in G$ fixing a vertex are contained in the kernel of $\Theta$.

\section{Actions of Cremona groups on Median graphs}
\noindent
Let $X$ be a normal variety over a field $\kk$, i.e.\ an integral and separated scheme of finite type. We will construct a median graph $\mathcal{C}^0(X)$ with an isometric action of $\Bir(X)$. We define the vertices of $\mathcal{C}^0(X)$ to be equivalence classes of marked pairs $(A,\varphi)$, where $A$ is a normal variety and $\varphi\colon A\dashrightarrow X$ is a birational map. We denote such a class by $[A,\varphi]$. Two marked pairs $(A,\varphi)$ and $(B,\psi)$ are equivalent if $\varphi^{-1}\psi\colon B\dashrightarrow A$ is an isomorphism in codimension $1$. 
Two vertices $v$ and $w$ are connected by an edge oriented from $v$ to $w$ if 
there exists a variety $A$ with a marking $\varphi\colon A\dashrightarrow X$ and hypersurface $H\subset A$ such that $v$ can be represented by $(A,\varphi)$ and $w$ by $(A\setminus H,\varphi|_{A\setminus H})$. 

\begin{theorem}[{\cite[Theorem~1.2]{lonjou-urech}}]\label{thm:medianC0}
	The graph $\mathcal{C}^0(X)$ is a median graph. 
\end{theorem}

\noindent
Let us describe the hyperplanes in $\mathcal{C}^0(X)$. An edge in $\mathcal{C}^0(X)$ is given by removing an irreducible hypersurface $H$ from a marked variety $(A,\varphi)$. We denote the hyperplane defined by this edge by $[(A,\varphi, H)]$. One can check the following:

\begin{lemma}[{\cite[Lemma~4.10]{lonjou-urech}}]\label{lemma_hyp}
	We have that $[(A,\varphi, H)]=[(B,\psi, K)]$ if and only if $K$ is not contained in the exceptional locus of $\varphi^{-1}\psi$ and $H$ is the strict transform $\varphi^{-1}\psi(K)$ of $K$.
\end{lemma} 

\noindent
We obtain an action of $\Bir(X)$ on $\mathcal{C}^0(X)$ by letting an $f\in\Bir(X)$ map a vertex $[A,\varphi]$ of $\mathcal{C}^0(X)$ to the vertex $[A,f\varphi]$. It is straightforward to check that this is well defined and defines an action of $\Bir(X)$ on $\mathcal{C}^0(X)$ without inversion. Let us also note that an element $f\in\Bir(X)$ fixes a vertex $[Y,\psi]$ in $\mathcal{C}^0(X)$ if and only if $[Y,\psi]=[Y,f\psi]$, which means that $\psi^{-1}f\psi$ is an automorphism in codimension 1 of $Y$, i.e.\ $f$ is pseudo-regularisable.

\begin{lemma}
	The map $ [(X,\id,H)]\mapsto H $ induces a bijection between the set of $\Bir(X)$-orbits of hyperplanes of the form $[(X,\id,H)]$ and the prime divisors of $X$ modulo Cremona equivalence.
\end{lemma}
\begin{proof}
Two hyperplanes $[(X,\varphi, H)]$ and $[(X,\psi, K)]$ are in the same $\Bir(X)$-orbit if and only if, by Lemma \ref{lemma_hyp}, there exists a birational transformation $g\in\Bir(X)$ such that the strict transform $\psi^{-1}g\varphi(H)$ equals $K$, which is equivalent to the existence of a birational transformation $\tilde{g}\in \Bir(X)$ such that  the strict transform $\tilde{g}(H)$ equals $K$, i.e.\ the two irreducible hypersurfaces $H$ and $K$ are {Cremona equivalent}. 
\end{proof}

\subsection{Proof of Theorem~\ref{thm:quotient}}
We are now ready to deduce  Theorem~\ref{thm:quotient} with the help of the homomorphism $\Theta$ constructed above. The  hyperplanes of $\mathcal{C}^0(X)$ already come with an orientation. Moreover, hyperplanes are pairwise transverse (\cite[Proposition~4.7]{lonjou-urech}), so the graph $\Gamma$ constructed in Section~\ref{sec:hom} is complete and the group $A(\Gamma)$ is therefore free abelian.

\medskip\noindent
Fix the vertex $o=[X,\id]$ in $\mathcal{C}^0(X)$. We obtain a path from $o$ to $f(o)$ defined by the vertices  
\[[X,\id], [X\setminus K_1,\id],\dots,[X\setminus \{K_1\cup\dots\cup K_m\},\id]=[X\setminus \{H_1\cup\dots\cup H_k\},f],\dots, [X,f],\]
This shows that the image $\Theta(\Bir(X))$ is in fact contained in $A(\Gamma')$, where $\Gamma'\subset \Gamma$ is the complete subgraph whose vertices are represented by hyperplanes of the form $[(X,\psi, H)]$. Moreover, since two hyperplanes $[(X,\psi, H)]$ and $[(X,\psi', H')]$ represent the same vertex if and only if $H$ and $H'$ are Cremona equivalent by Lemma \ref{lemma_hyp}, we may identify the vertices of $\Gamma'$ with the Cremona equivalence classes of hypersurfaces in $X$ and the group $A(\Gamma')$ with $\Z[\Div(X)/_\approx]$. The path $[o, f(o)]$ crosses the hypersurface classes $[K_j]$ with positive sign and the hypersurface classes $[H_i]$ with negative sign. This proves the first part of the theorem. Clearly, an $f$ satisfying the conditions from the second part of the theorem is not contained in the kernel. Finally, every pseudo-regularisable element fixes a vertex in $\mathcal{C}^0(X)$. Hence, all pseudo-regularisable elements are in the kernel of $\Theta$. \qed

\bibliographystyle{amsalpha}
{\footnotesize\bibliography{bibliography_cu}}
\Address

\end{document}